\documentclass[11pt,a4paper]{article}
\usepackage{amsfonts,amsgen,amsmath,amstext,amsbsy,amsopn,amsthm,amsfonts,amssymb,amscd}
\usepackage{epsf,epsfig}
\usepackage{float}
\usepackage{ebezier,eepic}
\usepackage{color}
\usepackage{tikz}
\usepackage{multirow}
\usepackage{graphicx}
\usepackage{subfigure}
\newtheorem{thm}{Theorem}[section]

\newtheorem{lem}[thm]{Lemma}
\newtheorem{false statement}{False statement}

\theoremstyle{definition}

\newtheorem{claim}{Claim}

\makeatletter \@addtoreset{equation}{section}

\def\hn{\mathcal{N}}
\def\hf{\mathcal{F}}
\def\hk{\mathcal{K}}
\begin{document}
\title{\bf\Large The Generalized Tur\'{a}n Problem of Two Intersecting Cliques}
\date{}
\author{ Erica L.L. Liu$^1$, Jian Wang$^2$\\[10pt]
$^{1}$Center for Applied Mathematics\\
Tianjin University\\
Tianjin 300072, P. R. China\\[6pt]
$^{2}$Department of Mathematics\\
Taiyuan University of Technology\\
Taiyuan 030024, P. R. China\\[6pt]
E-mail:  $^1$liulingling@tju.edu.cn, $^2$wangjian01@tyut.edu.cn
}

\maketitle

\begin{abstract}
 For $s<r$, let $B_{r,s}$ be the graph consisting of two copies of $K_r$, which share exactly $s$ vertices. Denote by $ex(n, K_r, B_{r,s})$ the maximum number of copies of $K_r$ in a $B_{r,s}$-free graph on $n$ vertices.  In 1976, Erd\H{o}s and S\'{o}s determined $ex(n,K_3,B_{3,1})$. Recently, Gowers and Janzer showed that
$ex(n,K_r,B_{r,r-1})=n^{r-1-o(1)}$. It is a natural question to ask for $ex(n,K_r,B_{r,s})$ for general $r$ and $s$. In this paper, we mainly consider the problem for $s=1$. Utilizing the Zykov's symmetrization, we show that $ex(n,K_4, B_{4,1})=\lfloor (n-2)^2/4\rfloor$ for $n\geq 45$. For $r\geq 5$ and  $n$ sufficiently large, by the F\"{u}redi's structure theorem we show that $ex(n,K_r,B_{r,1}) =\mathcal{N}(K_{r-2},T_{r-2}(n-2))$, where $\mathcal{N}(K_{r-2},T_{r-2}(n-2))$ represents the number of copies of $K_{r-2}$ in the $(r-2)$-partite Tur\'{a}n graph on $n-2$ vertices.
\end{abstract}


\medskip

\section{Introduction}
 Let $T$ be a graph and $\hf$ be a family of graphs. We say that a graph $G$ is $\hf$-free if it does not contain any graph from $\hf$ as a subgraph.
Let $ex(n, T, \hf)$ denote the maximum possible number of copies of $T$ in an $\hf$-free graph on $n$ vertices. The problem of determining $ex(n,T,\mathcal{F})$ is often called the generalized Tur\'{a}n problem.  When $T=K_2$, it reduces to the classical Tur\'{a}n number $ex(n,\mathcal{F})$. For simplicity, we often write $ex(n,T, F)$ for $ex(n,T,\{F\})$.

Let $T$ be a graph on $t$ vertices. The $s$-blow-up of $T$ is the graph obtained by replacing each vertex $v$ of $T$ by
an independent set $W_v$ of size $s$, and each edge $uv$ of $T$ by a complete bipartite graph
between the corresponding two independent sets $W_u$ and $W_v$. Alon and Shikhelman \cite{alon} showed that $ex(n,T,F) =\Theta(n^t)$ if and only if for any positive integer $s$, $F$ is not a subgraph of the $s$-blow-up of $T$. Otherwise, there exists some $\epsilon(T,F)>0$ such that $ex(n,T,F) \leq n^{t-\epsilon(T,F)}$.

For integers $s<r$, let $B_{r,s}$ be the graph consisting of two copies of $K_r$, which share exactly $s$ vertices. In 1976, Erd\H{o}s and S\'{o}s \cite{sos} determined the maximum number of hyperedges in a 3-uniform hypergraph without two hyperedges intersecting in exactly one vertex.  From their result, it is easy to deduce the following theorem.
\begin{thm}[Erd\H{o}s and S\'{o}s \cite{sos}]
	For all $n$,
	\begin{align*}
	ex(n,K_3,B_{3,1})=\left\{
	\begin{array}{ll}
	n, & \hbox{$n\equiv0\pmod 4$;} \\
	n-1, & \hbox{$n\equiv1\pmod 4$;} \\
	n-2, & \hbox{$n\equiv2$ or $3\pmod 4$.}
	\end{array}
	\right.
	\end{align*}
\end{thm}

 The celebrated Ruzsa-Szemer\'{e}di theorem \cite{ruzsa} implies that $ex(n, K_3, B_{3,2})=n^{2-o(1)}$. Recently, Gowers and Janzer \cite{gowera} proposed a natural generalization of Ruzsa-Szemer\'{e}di Theorem, and proved the following result.
\begin{thm}[Gowers and Janzer \cite{gowera}]\label{gowers}
For each $2\leq s< r$,
\[
ex(n,K_r,\{B_{r,s},B_{r,s+1},\ldots,B_{r,r-1}\})=n^{s-o(1)}.
\]
\end{thm}

For a graph $G$, let $V(G)$ and $E(G)$ be the vertex set and edge set of $G$, respectively.  The {\it join} of two graphs $G_1$ and $G_2$, denoted by $G_1\vee G_2$, is defined as $V(G_1\vee G_2)=V(G_1)\cup V(G_2)$ and $E(G_1\vee G_2)=E(G_1)\cup E(G_2)\cup \{xy\colon x\in V(G_1), y\in V(G_2)\}$. The $r$-partite Tur\'{a}n graph on $n$ vertices, denoted by $T_r(n)$, is a complete $r$-partite graph with each part of size differ by at most one. Denote by $\mathcal{N}(T,G)$ the number of copies of $T$ in $G$.

In this paper,  by Zykov's symmetrization \cite{zykov} we determine $ex(n, K_4, B_{4,1})$ for $n\geq 45$.

\begin{thm}\label{thm1}
For $n\geq 45$,
\begin{align*}
ex(n,K_4, B_{4,1})=\left\lfloor\frac{(n-2)^2}{4}\right\rfloor,
\end{align*}
and $K_2\vee T_2(n-2)$ is the unique graph attaining the maximum number of copies of $K_4$.
\end{thm}

Then, by F\"{u}redi's structure theorem \cite{furedi}, we determine $ex(n, K_r, B_{r,1})$ for $r\geq 5$ and $n$ sufficiently large.
\begin{thm}\label{thm3}
For $r\geq 5$ and sufficiently large $n$,
\begin{align*}
ex(n,K_r, B_{r,1})=\mathcal{N}(K_{r-2}, T_{r-2}(n-2)),
\end{align*}
and $K_2\vee T_{r-2}(n-2)$ is the unique graph attaining the maximum number of copies of $K_r$.
\end{thm}

Note that $B_{r,0}$ represents the union graph of two disjoint copies of $K_r$. By F\"{u}redi's structure theorem, we determine $ex(n, K_r, B_{r,0})$ for $r\geq 3$ and $n$ sufficiently large.

\begin{thm}\label{thm8}
For $r\geq 3$ and sufficiently large $n$,
\begin{align*}
ex(n, K_r, B_{r,0})=\mathcal{N}(K_{r-1},T_{r-1}(n-1)),
\end{align*}
and $K_1\vee T_{r-1}(n-1)$ is the unique graph attaining the maximum number of copies of $K_r$.
\end{thm}

Let $r,s$ be positive integers with $s<r$. An integer vector $(a_1,a_2,\ldots,a_t)$ is called a {\it partition} of $r$ if $a_1\geq a_2\geq \ldots\geq a_t>0$ and $\sum_{i=1}^t a_i=r$. Let $P=(a_1,a_2,\ldots,a_t)$ be a partition of $r$. If $\sum_{i\in I} a_i\neq s$ holds for every $I\subset \{1,2,\ldots,t\}$, then we call $P$ an {\it $s$-sum-free} partition of $r$. Denote by $\beta_{r,s}$ the maximum length of an $s$-sum-free partition of $r$.

\begin{thm}\label{thm9}
For any $r>s\geq 2$, if $r\geq 2s+1$,
\begin{align*}
ex(n, K_r, B_{r,s})=\Theta(n^{r-s-1});
\end{align*}
if $r\leq 2s$, then there exist positive reals $c_1$ and $c_2$ such that
\begin{align*}
c_1n^{\beta_{r,s}}\leq ex(n, K_r, B_{r,s})\leq c_2n^{s}.
\end{align*}
\end{thm}

Utilizing the graph removal lemma, we establish an upper bound on $ex(n,K_4,B_{4,2})$.

\begin{thm}\label{thm2}
For sufficiently large $n$,
\begin{align*}
\frac{n^2}{12}-2\leq ex(n,K_4, B_{4,2})\leq\frac{n^2}{9}+o(n^2).
\end{align*}
\end{thm}

The rest of this paper is organized as follows. In Section 2, we prove Theorem \ref{thm1}. In Section 3, we prove Theorems \ref{thm3} and \ref{thm8}. In Section 4, we prove Theorem \ref{thm9}. In Section 5, we prove Theorem \ref{thm2}.

\section{The value of $ex(n,K_4,B_{4,1})$}

Zykov \cite{zykov} introduced a useful tool to prove Tur\'{a}n theorem, which is called Zykov's symmetrization. In this section, by Zykov's symmetrization we first determine $ex(n,K_4,\{B_{4,1},$\\$H_1,K_5\})$, where $H_1$ is a graph on seven vertices as shown in Figure \ref{fig:fig}. Then, we show that a $B_{4,1}$-free graph can be reduced to a $\{B_{4,1},H_1,K_5\}$-free graph by deleting vertices, which leads to a proof of Theorem \ref{thm1}.

\begin{figure}[h]
	\centering
\begin{tikzpicture}
\begin{scope}[scale=2]
\filldraw (3,0) circle (0.8pt);
\filldraw (3.8,0) circle (0.8pt);
\filldraw (4.6,0) circle (0.8pt);
\filldraw (3.8,0.5) circle (0.8pt);
\filldraw (3.3,0.8) circle (0.8pt);
\filldraw (4.3,0.8) circle (0.8pt);
\filldraw (3.8,1.2) circle (0.8pt);
\draw (3,0)--(3.8,0)--(4.6,0)--(3.8,0.5)--(3,0);
\draw (3.8,0)--(3.8,0.5);
\draw (3.8,0.5)--(3.3,0.8)--(3,0);
\draw (3.3,0.8)--(3.8,0);
\draw (4.3,0.8)--(3.8,0);
\draw (3.8,0.5)--(4.3,0.8)--(4.6,0);
\draw (3.3,0.8)--(3.8,1.2)--(4.3,0.8)--(3.3,0.8);
\draw (3.8,1.2)--(3.8,0.5);
\end{scope}
\end{tikzpicture}
	\caption{A graph $H_1$ on seven vertices. }
	\label{fig:fig}
\end{figure}
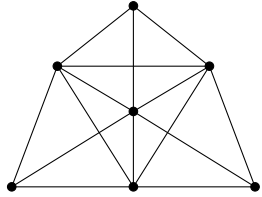

For $S\subset V(G)$, let $G[S]$ denote the subgraph of $G$ induced by $S$, and let $G - S$ denote the subgraph of $G$ induced by $V(G)\setminus S$.

\begin{lem}\label{lem-1}
For $n\geq 2$,
\begin{align*}
ex(n,K_4, \{B_{4,1}, H_1, K_5\})=\left\lfloor\frac{(n-2)^2}{4}\right\rfloor,
\end{align*}
and $K_2\vee T_2(n-2)$ is the unique graph attaining the maximum number of $K_4$'s.
\end{lem}

\begin{proof}
Assume that $G$ is a $\{B_{4,1}, H_1, K_5\}$-free graph with the maximum number of copies of $K_4$.  We may further assume that  each edge of $G$ is contained in at least one copy of $K_4$, since otherwise we can delete it without decreasing the number of copies of $K_4$. For each $e\in E(G)$, let $\mathcal{K}_4(e)$ denote the set of copies of $K_4$ in $G$ containing $e$. Let
\[
E_1= \left\{e\in E(G)\colon \mbox{there exist }K,K'\in \mathcal{K}_4(e)\mbox{ such that }E(K)\cap E(K')=\{e\}\right\}
\]
and let $G_1$ be the subgraph of $G$ induced by $E_1$.

\begin{claim}\label{claim3}
$E_1$ is a matching of $G$.
\end{claim}
\begin{proof}
Suppose to the contrary that there exists a path of length two in $G_1$, say $vuw$.  Since $uv\in E_1$, there exist distinct vertices $a_1,b_1,a_2,b_2$ so that both $G[\{u,v,a_1,b_1\}]$ and $G[\{u,v,a_2,b_2\}]$ are copies of $K_4$. Since $uw\in E_1$, there exist distinct vertices $c_1,d_1,c_2,d_2$ so that both $G[\{u,w,c_1,d_1\}]$ and $G[\{u,w,c_2,d_2\}]$ are copies of $K_4$.

{ Case 1.} $w\in\{a_1,b_1,a_2,b_2\}$ or $v\in\{c_1,d_1,c_2,d_2\}$.

Since the two cases are symmetric, we only consider the case $w\in\{a_1,b_1,a_2,b_2\}$. By symmetry, we may assume that $a_1=w$. Now $G[\{u,v,w,b_1\}]$ and $G[\{u,v,a_2,b_2\}]$ are both copies of $K_4$. Clearly, we have either $v\notin \{c_1,d_1\}$ or $v\notin \{c_2,d_2\}$. Without loss of generality, assume that $v\notin\{c_1,d_1\}$. If $\{c_1,d_1\}\cap\{a_2,b_2\}=\emptyset$, then $G[\{u,v,w,a_2,b_2,c_1,d_1\}]$ contains a copy of $B_{4,1}$, which contradicts the assumption that $G$ is $B_{4,1}$-free. If $|\{c_1,d_1\}\cap\{a_2,b_2\}|=1$, by symmetry we assume that $c_1=a_2$, then $G[\{u,v,w,b_1,a_2,b_2,d_1\}]$ contains a copy of $H_1$, a contradiction. If $\{c_1,d_1\}=\{a_2,b_2\}$, then $G[\{u,v,w,a_2,b_2\}]$ is a copy of $K_5$, a contradiction.

{ Case 2.}  $w\notin\{a_1,b_1,a_2,b_2\}$ and $v\notin\{c_1,d_1,c_2,d_2\}$.

For $i,j\in \{1,2\}$, we claim  that $|\{a_i,b_i\}\cap\{c_j,d_j\}|=1$. If $\{a_i,b_i\}\cap \{c_j,d_j\}=\emptyset$, then $G[\{u,v,w,a_i,b_i,c_j,d_j\}]$ contains $B_{4,1}$ as a subgraph, a contradiction. If $\{a_i,b_i\}=\{c_j,d_j\}$, then $G[\{u,v,w,a_i,b_i,c_i,d_i\}]$ contains $B_{4,1}$ as a subgraph, a contradiction. Hence $|\{a_i,b_i\}\cap\{c_j,d_j\}|=1$. It follows that $\{a_1,b_1,a_2,b_2\}=\{c_1,d_1,c_2,d_2\}$. Then $G[\{u,v,w,a_1,b_1,a_2,b_2\}]$ contains $H_1$ as a subgraph, a contradiction. Thus, the claim holds.\end{proof}

Let $G_2=G- V(G_1)$. For two distinct vertices $u,v\in V(G)$ with $uv\notin E(G)$, define  $C_{uv}(G)$ to be the graph obtained by deleting edges incident to $u$ and adding edges in $\{uw\colon w\in N(v)\}$.

\begin{claim}\label{claim22}
For two distinct vertices $u,v\in V(G_2)$ with $uv\notin E(G)$,  $C_{uv}(G)$ is a $\{B_{4,1}, H_1, K_5\}$-free graph.
\end{claim}

\begin{proof}
Let $\tilde{G} = C_{uv}(G)$. Since $uv\notin E(G)$, clearly we have $uv\notin E(\tilde{G})$. We first claim that $\tilde{G}$ is $K_5$-free. Otherwise, since $G$ is $K_5$-free, there is a vertex set $K$ containing $u$ such that $\tilde{G}[K]\cong K_5$. Then $v\notin K$ since $uv\notin E(\tilde{G})$. It follows that $K\setminus\{u\}\cup \{v\}$ induces a copy of $K_5$ in $G$, a contradiction.

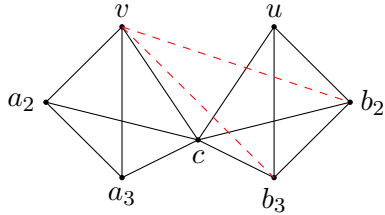
\begin{figure}[h]
	\centering
\begin{tikzpicture}[scale=1]
\filldraw (0,0) circle (0.8pt);
\filldraw (1,1) circle (0.8pt);
\filldraw (1,-1) circle (0.8pt);
\filldraw (2,-0.5) circle (0.8pt);
\filldraw (3,1) circle (0.8pt);
\filldraw (3,-1) circle (0.8pt);
\filldraw (4,0) circle (0.8pt);

\draw(0,0) node[left] {$a_2$};\draw(1,1) node[above] {$v$};
\draw(1,-1) node[below] {$a_3$};\draw(2,-0.5) node[below] {$c$};\draw(3,1) node[above] {$u$};\draw(3,-1) node[below] {$b_3$};
\draw(4,0) node[right] {$b_2$};
\draw (2,-0.5) -- (1,1) -- (0,0) --(1,-1)--(2,-0.5)--(3,1)--(4,0)--(3,-1)--(2,-0.5);
\draw (0,0)--(2,-.5)--(4,0);\draw (1,1)--(1,-1);\draw (3,1)--(3,-1);
\draw[dashed, red] (1,1) -- (3,-1);
\draw[dashed, red] (1,1) -- (4,0);
\end{tikzpicture}

	\caption[]{A copy of $B_{4,1}$ in $\tilde{G}$.}
	\label{fig:fz}
\end{figure}

If $\tilde{G}$ contains a copy of $B_{4,1}$, let $S=\{a_1,a_2,a_3,b_1,b_2,b_3,c\}$ be a subset of $V(\tilde{G})$ such that both $\tilde{G}[\{a_1,a_2,a_3,c\}]$ and $\tilde{G}[\{b_1,b_2,b_3,c\}]$ are copies of $K_4$. If $u\notin S$, then $G[S]$ is a copy of $B_{4,1}$, a contradiction. If $u\in S$ but $v\notin S$, then $G[(S\setminus\{u\})\cup \{v\}]$ is a copy of $B_{4,1}$, a contradiction. If $u,v\in S$, since $uv\notin E(\tilde{G})$, by symmetry we may assume that $a_1=v$ and $b_1=u$.
Since $u$ is a ``clone" of $v$ in $\tilde{G}$, we have $vb_2,vb_3\in E(G)$ (as shown in Figure \ref{fig:fz}). Then both $G[\{v,c,a_2,a_3\}]$ and $G[\{v,c,b_2,b_3\}]$ are copies of $K_4$ in $G$. It follows that $vc$ is an edge in $E_1$ in $G$, which contradicts the assumption that $v\in V(G)\setminus V(G_1)$. Thus $\tilde{G}$ is $B_{4,1}$-free.

\begin{figure}[h]
	\centering
\begin{tikzpicture}
\begin{scope}[scale=2]
\filldraw (3,0) circle (0.8pt);\draw(3,0) node[below] {$l$};
\filldraw (3.8,0) circle (0.8pt);\draw(3.8,0) node[below] {$m$};
\filldraw (4.6,0) circle (0.8pt);\draw(4.6,0) node[below] {$n$};
\filldraw (3.8,0.5) circle (0.8pt);\draw(3.8,0.5) node[right] {$k$};
\filldraw (3.3,0.8) circle (0.8pt);\draw(3.3,0.8) node[below] {$i$};
\filldraw (4.3,0.8) circle (0.8pt);\draw(4.3,0.8) node[below] {$j$};
\filldraw (3.8,1.2) circle (0.8pt);\draw(3.8,1.2) node[above] {$h$};
\draw (3,0)--(3.8,0)--(4.6,0)--(3.8,0.5)--(3,0);
\draw (3.8,0)--(3.8,0.5);
\draw (3.8,0.5)--(3.3,0.8)--(3,0);
\draw (3.3,0.8)--(3.8,0);
\draw (4.3,0.8)--(3.8,0);
\draw (3.8,0.5)--(4.3,0.8)--(4.6,0);
\draw (3.3,0.8)--(3.8,1.2)--(4.3,0.8)--(3.3,0.8);
\draw (3.8,1.2)--(3.8,0.5);
\end{scope}
\end{tikzpicture}
	\caption{A copy of $H_1$ in $\tilde{G}$. }
	\label{fig:fig-3}
\end{figure}
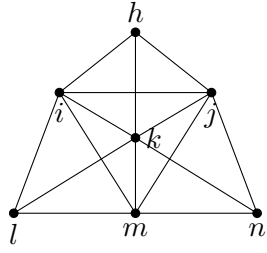

If $\tilde{G}$ contains a copy of $H_1$, let $T=\{h,i,j,k,l,m,n\}$ be a subset of $V(\tilde{G})$ such that $\tilde{G}[\{h,i,j,k\}]$, $\tilde{G}[\{i,j,k,m\}]$, $\tilde{G}[\{i,k,l,m\}]$ and $\tilde{G}[\{j,k,m,n\}]$ are all copies of $K_4$ as shown in Figure \ref{fig:fig-3}. Similarly, we have $u,v\in T$. Since $uv\notin E(\tilde{G})$, by symmetry we have to consider three cases: (i) $h=u$, $n=v$; (ii) $h=u$, $m=v$ or (iii) $h=v$, $m=u$. If $h=u$ and $n=v$, then $vi\in E(G)$ since $ui\in E(\tilde{G})$. It follows that $\{i,j,k,m,v\}$ induces a copy of $K_5$ in $G$, which contradicts the assumption that $G$ is $K_5$-free. If $h=u$ and $m=v$, then $kv\in E_1$ since both $G[\{k,v,i,l\}]$ and $G[\{k,v,j,n\}]$ are copies of $K_4$, which contradicts the fact that $v\in V(G_2)$. If $h=v$ and $m=u$, then $vl,vn\in E(G)$ since $ul,un\in E(\tilde{G})$. It follows that both $G[\{k,v,i,l\}]$ and $G[\{k,v,j,n\}]$ are copies of $K_4$, which contradicts the fact that $v\in V(G_2)$. Hence $\tilde{G}$ is $H_1$-free.
\end{proof}

By Zykov symmetrization, we prove the following claim.

\begin{claim}\label{claim4}
$G_2$ is a complete $r$-partite graph with $r\leq 4$.
\end{claim}
\begin{proof}  Recall that  $G$ is a $\{B_{4,1}, H_1, K_5\}$-free graph with the maximum number of copies of $K_4$ and  each edge of $G$ is contained in at least one copy of $K_4$. We define a binary relation $R$ in $V(G_2)$ as follows: for any two vertices $x,y\in V(G_2)$, $xRy$ if and only if $xy\notin E(G)$. We shall show that $R$ is an equivalence relation.  Since $G$ is loop-free, it follows that $R$ is reflexive. Since $G$ is a undirected graph, it follows that $R$ is symmetric.

Now we show that $R$ is transitive. Suppose to the contrary that there exist $x,y,z\in V(G_2)$ such that $xy, yz\notin E(G_2)$ but $xz\in E(G_2)$. For $u,v\in V(G_2)$, let $k_4(u)$ be the number of copies of $K_4$ in $G$ containing $u$, and $k_4(u,v)$ be the number of copies of $K_4$ in $G$ containing $u$ and $v$.

{ Case 1.}  $k_4(y)<k_4(x)$ or $k_4(y)<k_4(z)$. Since the two cases are symmetric, we only consider the case $k_4(y)<k_4(x)$. Let $\tilde{G}=C_{yx}(G)$. By Claim \ref{claim22}, $\tilde{G}$ is $\{B_{4,1}, H_1, K_5\}$-free since $G$ is $\{B_{4,1}, H_1, K_5\}$-free.   But now we have
\[
\hn(K_4,\tilde{G}) = \hn(K_4,G) -k_4(y) +k_4(x)> \hn(K_4,G),
\]
which contradicts the assumption that $G$ is a $\{B_{4,1}, H_1, K_5\}$-free graph with the maximum number of copies of $K_4$.

{ Case 2.} $k_4(y)\geq k_4(x)$ and $k_4(y)\geq k_4(z)$. Let $G^*=C_{xy}(C_{zy}(G))$. By Claim \ref{claim22}, $G^*$ is $\{B_{2,1}, H_1, K_5\}$-free. Since each edge in $G$ is contained in at least one copies of $K_4$, it follows that
\begin{align*}
\mathcal{N}(K_4, G^*)&=\mathcal{N}(K_4, G)-(k_4(x)+k_4(z)-k_4(x,z))+2k_4(y)\\
&\geq \mathcal{N}(K_4, G)+k_4(x,z)\\
&>\mathcal{N}(K_4, G),
\end{align*}
which contradicts the assumption that $G$ is a $\{B_{4,1}, H_1, K_5\}$-free graph with the maximum number of copies of $K_4$. Thus, we conclude that $xz\notin E(G)$ and $R$ is transitive. Since $R$ is an equivalence relation on $V(G_2)$ and $G$ is $K_5$-free, it follows that $G_2$ is a complete $r$-partite graph with $r\leq 4$.
\end{proof}

\begin{claim}\label{claim5}
For any copy $K$ of $K_4$ in $G$ and any $uv\in E_1$,  $|V(K)\cap \{u,v\}|\neq 1$.
\end{claim}

\begin{proof}
Suppose for contradiction that there exists $\{a, b,c,d,v\}\subset V(G)$ such that $G[\{a,b,c,d\}]$ is isomorphic to $K_4$ and $bv$ is an edge in $E_1$, as shown in Figure \ref{fig:f3}.

\begin{figure}[h]
	\centering
\begin{tikzpicture}[scale=1.5]
\filldraw (0,1) circle (0.8pt);
\filldraw (0,0) circle (0.8pt);
\filldraw (1,0) circle (0.8pt);
\filldraw (1,1) circle (0.8pt);
\filldraw (2,1.3) circle (0.8pt);
\draw(0,1) node[left] {$a$};\draw(0,0) node[left] {$d$};
\draw(1,1) node[above] {$b$};\draw(1,0) node[right] {$c$};
\draw(2,1.3) node[right] {$v$};
\draw (0,0) -- (0,1) -- (1,1) --(1,0)--(0,0)--(1,1)--(2,1.3);
\draw (0,1)--(1,0);
\end{tikzpicture}
	\caption{An edge in $E_1$ is attached to a copy of $K_4$.}
	\label{fig:f3}
\end{figure}
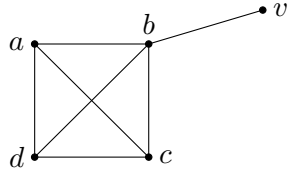

Since $bv\in E_1$, there exist distinct vertices $x_1,y_1,x_2,y_2$ such that both $G[\{b,v,x_1,y_1\}]$ and $G[\{b,v,x_2,y_2\}]$ are copies of $K_4$ in $G$. Then either  $|\{x_1,y_1\}\cap \{a,c,d\}|\leq 1$ or $|\{x_2,y_2\}\cap \{a,c,d\}|\leq 1$ holds since  $x_1,y_1,x_2,y_2$ are distinct. By symmetry, we assume that  $|\{x_1,y_1\}\cap \{a,c,d\}|\leq 1$. If $\{x_1,y_1\}\cap \{a,c,d\}=\emptyset$, then $G[\{b,v,x_1,y_1,a,c,d\}]$ contains a copy of $B_{4,1}$, a contradiction. If $|\{x_1,y_1\}\cap \{a,c,d\}|=1$, without loss of generality, we assume that $x_1=a$. Since both $G[\{a,b,x_2,v\}]$ and $G[\{a,b,c,d\}]$ are copies of $K_4$, it follows that $ab\in E_1$, which contradicts Claim \ref{claim3}. Thus, we conclude that $|V(K)\cap \{u,v\}|\neq 1$ for any copy $K$ of $K_4$ in $G$ and any $uv\in E_1$.
\end{proof}

Now let $K$ be a copy of $K_4$ in $G$. Recall that $E_1$ is a matching in $G$ and $G_1$ is the graph induced by $E_1$. If $|V(K)\cap V(G_1)|=1$ or 3, then we will find an edge in $E_1$ attached to $K$, which contradicts  Claim \ref{claim5}. Thus $|V(K)\cap V(G_1)|\in \{0,2,4\}$. Moreover, if $|V(K)\cap V(G_1)|=2$, let $\{x,y\}= V(K)\cap V(G_1)$, then by Claim \ref{claim5} we have  $xy\in E_1$. Recall that  $\mathcal{K}_4(e)$ represents the set of copies of $K_4$ in $G$ containing $e$ for $e\in E(G)$. Define
\begin{align*}
\hk_1(G)=&\{K\colon K\mbox{ is a copy of }K_4 \mbox{ in } G \mbox{ and } V(K)\subset V(G_1)\};\\
\hk_2(G)=&\{K\colon K\mbox{ is a copy of }K_4 \mbox{ in } G \mbox{ and } V(K)\subset V(G_2)\};\\
\hk_3(G)=&\{K\colon K \in \hk_4(e) \mbox{ for some }e \in E_1\mbox{ and } |V(K)\cap V(G_1)|=2\}.
\end{align*}

Let $|V(G_1)|=n_1$, $|V(G_2)|=n-n_1=n_2$. Since $E_1$ is a matching, it follows that $n_1$ is even. By Claim \ref{claim5}, for any $K\in \hk_1(G)$ we have $E(K)\cap E_1$ is a matching of size 2. To derive an upper bound on $|\hk_1(G)|$, we define a graph $H$ with  $V(H)=E_1$ as follows. For any $e_1,e_2\in E_1$, $e_1e_2$ is an edge of $H$ if and only if there exists a copy of $K_4$ containing both $e_1$ and $e_2$. Since $G$ is $K_5$-free, it is easy to see that $H$ is triangle-free. Moreover, each copy of $K_4$ in $G$  corresponds to an edge in $H$. Thus, by Mantel's Theorem \cite{mantel} we have
\[
|\hk_1(G)|= e(H) \leq \left\lfloor\frac{|E_1|^2}{4}\right\rfloor=\left\lfloor\frac{n_1^2}{16}\right\rfloor.
\]

We have shown that $G_2$ is a complete $r$-partite graph with $r\leq 4$ in Claim \ref{claim4}. If $r\leq 1$, then $\hk_2(G)= \hk_3(G)=\emptyset$. Thus, we have
\[
\hn(K_4,G)=|\hk_1(G)|\leq \left\lfloor\frac{n_1^2}{16}\right\rfloor \leq \left\lfloor\frac{n^2}{16}\right\rfloor \leq \left\lfloor\frac{(n-2)^2}{4}\right\rfloor,
\]
where the equalities hold if and only if $n=4$ and $G$ is isomorphic to $K_4$.

If $r=2$, then $\hk_2(G)=\emptyset$. If $n_1=0$, then we have $\hn(K_4,G)=0$. Hence we may assume that $n_1\geq 2$. We claim that  each edge in $E(G_2)$ is contained in at most one copy of $K_4$ in $\hk_3(G)$.  Otherwise, by the definition of $\hk_3(G)$, there exists an edge $e\in E(G_2)$ contained in two distinct copies of $K_4$, which contradicts the fact that $e\notin E_1$. Then
\[
|\hk_3(G)| \leq e(G_2) \leq \left\lfloor\frac{n_2^2}{4}\right\rfloor.
\]
Thus, we have
\[
\hn(K_4,G)=|\hk_1(G)|+|\hk_3(G)| \leq \left\lfloor\frac{n_1^2}{16}\right\rfloor+\left\lfloor\frac{n_2^2}{4}\right\rfloor.
\]
For even integer $x$ with  $2\leq x\leq n$, let
\[
f(x) = \left\lfloor\frac{x^2}{16}\right\rfloor+\left\lfloor\frac{(n-x)^2}{4}\right\rfloor.
\]
Then
\begin{align*}
f(x-2)&=\left\lfloor\frac{(x-2)^2}{16}\right\rfloor+\left\lfloor\frac{(n-x+2)^2}{4}\right\rfloor\\
&=\left\lfloor\frac{x^2-4x+4}{16}\right\rfloor+\left\lfloor\frac{(n-x)^2}{4}+n-x+1\right\rfloor\\
&\geq \left\lfloor\frac{x^2}{16}\right\rfloor -\frac{x-1}{4}-1 +\left\lfloor\frac{(n-x)^2}{4}\right\rfloor+n-x+1\\
&\geq f(x)+n-\frac{5x-1}{4}
\end{align*}
and
\begin{align*}
f(x-2)&\leq \left\lfloor\frac{x^2}{16}\right\rfloor -\frac{x-1}{4}+1 +\left\lfloor\frac{(n-x)^2}{4}\right\rfloor+n-x+1\\
&\leq f(x)+n-\frac{5x-9}{4}.
\end{align*}
Thus, $f(x-2)\geq f(x)$ for $x\leq \frac{4n+1}{5}$ and $f(x-2)\leq f(x)$ for $x\geq \frac{4n+9}{5}$.
Therefore, for even $n$ we have
\[
\hn(K_4,G) \leq \max\{f(2),f(n)\}=\max\left\{\left\lfloor\frac{(n-2)^2}{4}\right\rfloor,\left\lfloor\frac{n^2}{16}\right\rfloor\right\}\leq \left\lfloor\frac{(n-2)^2}{4}\right\rfloor,
\]
where the equality holds if and only if $G$ is isomorphic to $K_2\vee T_2(n-2)$. For
odd $n$ we have
\[
\hn(K_4,G) \leq \max\{f(2),f(n-1)\}=\max\left\{\left\lfloor\frac{(n-2)^2}{4}\right\rfloor,\left\lfloor\frac{(n-1)^2}{16}\right\rfloor\right\}\leq \left\lfloor\frac{(n-2)^2}{4}\right\rfloor,
\]
where the equality holds if and only if $G$ is isomorphic to $K_2\vee T_2(n-2)$.


If $r= 3$, there exists a triangle $xyz$ in $G_2$. Since each edge in $G$ is contained in at least one copy of $K_4$, by Claim \ref{claim5} there exist $ab, cd\in E_1$ such that both $G[\{x,y,a,b\}]$ and $G[\{y,z,c,d\}]$ are copies of $K_4$ in $G$. Since $E_1$ is a matching, we have either $\{a,b\}=\{c,d\}$ or $\{a,b\}\cap \{c,d\}=\emptyset$.  If $\{a,b\}=\{c,d\}$, then $G[\{x,y,z,a,b\}]$ is a copy of $K_5$, a contradiction. If $\{a,b\}\cap \{c,d\}=\emptyset$, then then $G[\{x,y,z,a,b,c,d\}]$ contains  $B_{4,1}$, a contradiction. Thus, we conclude that $r\neq 3$.

If $r=4$, let $V_1,V_2,V_3,V_4$ be four vertex classes of $G_2$. Since $G$ is $B_{4,1}$-free,  at least two of $|V_i|$'s equal one. Without loss of generality, we assume that $|V_3|=|V_4|=1$. Let $V_3=\{u\}$ and $V_4=\{v\}$. Since $uv\notin E_1$, it follows that one of $|V_1|$ and $|V_2|$ equal one. By symmetry let $|V_2|=1$. Then, we have
\[
|\hk_2(G)| = |V_1| = n_2-3.
\]
Moreover, we claim that $\hk_3(G)=\emptyset$. Otherwise, assume that there exists $K\in\mathcal{K}_3(G)$ such that $V(K)\cap V(G_2)=\{x,y\}$. Since $x,y$ also contained in some $K'\in\hk_2(G)$, it follows that $E(K)\cap E(K')=\{xy\}$, which contradicts the fact that $xy\notin E_1$. Since $4\leq n_2\leq n$, we have
\begin{align*}
\hn(K_4,G) &=|\hk_1(G)|+|\hk_2(G)|\\
 &\leq  \left\lfloor\frac{n_1^2}{16}\right\rfloor + n_2-3\\
 &\leq \max\left\{\left\lfloor\frac{(n-4)^2}{16}\right\rfloor +1,n-3\right\}\\
 &\leq \left\lfloor\frac{(n-2)^2}{4}\right\rfloor,
\end{align*}
in which the equality holds if and only if $n=4$ and  $G\cong K_4$ or  $n=5$ and $G \cong K_2\vee T_2(3)$.
Thus, the lemma holds.
\end{proof}

Now we are in position to prove Theorem \ref{thm1}.

\begin{proof}[Proof of Theorem \ref{thm1}]
Let $G$ be a $B_{4,1}$-free graph on $n$ vertices. We show that $G$ can be made $\{B_{4,1},H_1,K_5\}$-free by deleting vertices. Let $H_2$ be a graph on six vertices as shown in Figure \ref{fig:fig-2}.

\begin{figure}[h]
	\centering
\begin{tikzpicture}
\begin{scope}[scale=2]
\filldraw (0,0.2) circle (0.8pt);\draw(0,0.2) node[left] {$b$};
\filldraw (0,0.8) circle (0.8pt);\draw(0,0.8) node[left] {$a$};
\filldraw (0.6,0) circle (0.8pt);\draw(0.6,0) node[below] {$c$};
\filldraw (0.6,1) circle (0.8pt);\draw(0.6,1) node[above] {$e$};
\filldraw (1.6,0.5) circle (0.8pt);\draw(1.6,0.5) node[right] {$f$};
\filldraw (1,0.5) circle (0.8pt);\draw(1,0.5) node[below] {$d$};
\draw (0,0.2) -- (0,0.8) -- (0.6,1) --(0.6,0)--(0,0.2)--(0.6,1)--(1,0.5)--(0.6,0)--(0,0.8);
\draw (0,0.2)--(1,0.5)--(0,0.8);
\draw (0.6,0)--(1.6,0.5)--(0.6,1);
\draw (1.6,0.5)--(1,0.5);
\end{scope}
\end{tikzpicture}
	\caption{A graph $H_2$ on six vertices. }
	\label{fig:fig-2}
\end{figure}
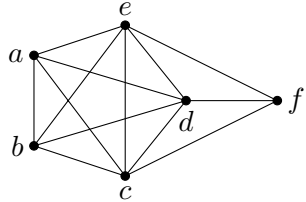

\begin{claim}\label{claim1}
There exists a subset $V'\subset V(G)$ such that $G'=G - V'$ is $\{H_1, H_2\}$-free and $\mathcal{N}(K_4, G')\geq \mathcal{N}(K_4,G) -10 |V'|$.
\end{claim}

\begin{proof}
 Assume that $G$ contains $H_2$ as a subgraph. Without loss of generality, we further assume that $A=\{a,b,c,d,e,f\}$ is a subset of $V(G)$ such that $G[A]$ contains $H_2$ (see Figure \ref{fig:fig-2}). We first claim that $V(K)\subset A$ for each copy $K$ of $K_4$ containing $f$. Otherwise, if $|V(K)\cap A|=1$, then $K$ and $G[\{c,d,e,f\}]$ are both  copies of $K_4$ that share exactly one vertex $f$, contradicting the fact that $G$ is $B_{4,1}$-free. If $|V(K)\cap A|=2$, by symmetry we may assume that $V(K)\cap A=\{e,f\}$. Then $K$ and $G[\{b,c,d,e\}]$ are both copies of $K_4$ that share exactly one vertex $e$, a contradiction. If $|V(K)\cap A|=3$, by symmetry we assume that $V(K)\cap A=\{d,e,f\}$. Then $K$ and $G[\{a,b,c,e\}]$ are both copies of $K_4$ that share exactly one vertex $e$, a contradiction. Thus, we conclude  $V(K)\subset A$ for each copy $K$ of $K_4$ containing $f$. Now we delete $f$ from $G$ to destroy a copy of $H_2$. By doing this, we loss at most $\binom{|A\setminus\{f\}|}{3}=10$ copies of $K_4$ since they are contained in $A$. We do it iteratively until the resulting graph is $H_2$-free. Let $G_1$ be the resulting graph and  $V_1$ be the set of deleted vertices. Clearly, we have $\mathcal{N}(K_4, G_1)\geq \mathcal{N}(K_4,G)-10|V_1|$.

 Now $G_1$ is $\{B_{4,1},H_2\}$-free. Assume that $G_1$ contains $H_1$ as a subgraph.  Let $B=\{h,i,j,k,l,m,n\}$ be a subset of $V(G_1)$ such that $G_1[B]$ contains $H_1$ (see Figure \ref{fig:fig-3}). It is easy to see that $hm$ is not an edge in $G_1$. Otherwise, $G_1[\{h,i,j,k,m\}]$ is a copy of $K_5$ and $G_1[\{h,i,j,k,m,l\}]$ contains a copy of $H_2$, a contradiction.  Now we claim that $V(K)\subset B\setminus\{m\}$ for each copy $K$ of $K_4$ in $G_1$ containing $h$. Otherwise, we have one of the following cases:
 \begin{itemize}
   \item If $V(K)\cap B\subset \{h,l,n\}$, then $K$ and $G_1[\{h,i,j,k\}]$ form a copy of $B_{4,1}$;
   \item if $|V(K)\cap\{i,j,k\}|=1$, then $K$ and $G_1[\{i,j,k,m\}]$ form a copy of $B_{4,1}$;
   \item if $V(K)\cap B=\{h,i,j\}$ or $\{h,i,k\}$, then $K$ and $G_1[\{j,k,m,n\}]$ form a copy of $B_{4,1}$;

   \item if $V(K)\cap B=\{h,j,k\}$, then $K$ and $G_1[\{i,k,l,m\}]$ form a copy of $B_{4,1}$.
 \end{itemize}
 Since $G_1$ is $B_{4,1}$-free, each of these cases leads to a contradiction. Note that $hm$ is not an edge in $G_1$. We have $V(K)\subset B\setminus\{m\}$ for each copy $K$ of $K_4$ in $G_1$ containing $h$. By deleting $h$ from $G_1$, we destroy a copy of $H_1$ and loss at most $\binom{|B\setminus\{h,m\}|}{3}=10$ copies of $K_4$. We do it iteratively until the resulting graph is $H_1$-free. Let $G_2$ be the resulting graph and  $V_2$ be the set of deleted vertices. Clearly, we have $\mathcal{N}(K_4, G_2)\geq \mathcal{N}(K_4,G_1) -10|V_2|$.

Let $G'=G_2$ and $V'=V_1\cup V_2$. Clearly, $G'$ is  $\{H_1,H_2\}$-free and
$\mathcal{N}(K_4, G')\geq \mathcal{N}(K_4,G)-10|V'|$.
\end{proof}

\begin{claim}\label{claim2}
There exists a subset $V''\subset V(G')$ such that $G''=G'- V''$ is $K_5$-free and $\mathcal{N}(K_4, G'')\geq \mathcal{N}(K_4,G') -|V''|$.
\end{claim}
\begin{proof}
Since $G'$ is $\{B_{4,1},H_2\}$-free, it is easy to see that each pair of copies of $K_5$ in $G'$ is vertex-disjoint. Let $T$ be a copy of $K_5$ in $G'$. We claim that $V(K)\subset V(T)$ for each copy $K$ of $K_4$ in $G'$ with  $V(T)\cap V(K)\neq \emptyset$. Otherwise, if $|V(K)\cap V(T)|\leq 2$, then it is easy to find a copy of $B_{4,1}$ in $G'$, a contradiction. If $|V(K)\cap V(T)|=3$, then we will find a copy of $H_2$ in $G'$, a contradiction. Thus, we conclude that $V(K)\subset V(T)$ for each copy $K$ of $K_4$ in $G'$ with  $V(T)\cap V(K)\neq \emptyset$. By deleting $V(T)$ from $G'$, we loss at most $\binom{|V(T)|}{4}=5$ copies of $K_4$.  Repeating this process,  finally we arrive at a $K_5$-free graph $G''$. Let $V''$ be the set of deleted vertices. Clearly, we have $G''$ is $K_5$-free and
$\mathcal{N}(K_4,G'')\geq \mathcal{N}(K_4, G')-|V''|$.
\end{proof}

Let $X=V'\cup V''$ and $|X|=x$. Note that $G''$ is $\{B_{4,1},H_1,K_5\}$-free. By Lemma \ref{lem-1} we have
\[
\hn(K_4,G'')\leq \left\lfloor\frac{(n-x-2)^2}{4}\right\rfloor.
\]
By Claims \ref{claim1} and \ref{claim2}, we have
\[
\hn(K_4,G)\leq \left\lfloor\frac{(n-x-2)^2}{4}\right\rfloor +10x =  \left\lfloor\frac{(n-x-2)^2}{4}+10x\right\rfloor.
\]
Since $f(x) = \frac{(n-x-2)^2}{4}+10x$ is a convex function and $0\leq x\leq n$, it follows that
\[
\hn(K_4,G)\leq \max\left\{\left\lfloor\frac{(n-2)^2}{4}\right\rfloor,10n+1\right\}.
\]
Since $n\geq 45$, we have $\mathcal{N}(K_4,G)\leq \lfloor(n-2)^2/4\rfloor$. Moreover, by Lemma \ref{lem-1}, the equality holds if and only if $G$ is isomorphic to $K_2\vee T_2(n-2)$. Thus, the theorem holds.
\end{proof}

\section{The values of $ex(n,K_r,B_{r,1})$ and $ex(n,K_r,B_{r,0})$}

By F\"{u}redi's structure theorem, Frankl and F\"{u}redi \cite{ff85} determined the maximum number of hyperedges in an $r$-uniform hypergraph without two hyperedges sharing exactly $s$ vertices for $r\geq 2s+2$. In this section, we determine $ex(n,K_r,B_{r,1})$ and $ex(n,K_r,B_{r,0})$ by following a similar approach.

First, we recall a result due to Frankl and F\"{u}redi in the intersection closed family (Lemma 5.5 in \cite{ff85}). Let $X$ be a finite set and $2^X$ be the family of all the subsets of $X$. We say that $\mathcal{I} \subset 2^X$ is {\it intersection closed} if for any $I, I'\in \mathcal{I} $, $I\cap I' \in \mathcal{I}$. We say $I\subset X$ is {\it covered} by $\mathcal{I}$ if there exists an $I'\in \mathcal{I}$ such that $I\subset I'$.

\begin{thm}[Frankl and F\"{u}redi \cite{ff85}]\label{thm5}
Let $r$ and $s$ be positive integers with $r\geq 2s+3$ and let $F$ be an $r$-element set.  Suppose that  $\mathcal{I}\subset 2^F\setminus\{F\}$ is an intersection closed family such that $|I|\neq s$ for any $I\in \mathcal{I}$ and all the $(r-s-2)$-element subsets of $F$ are covered by $\mathcal{I}$.  Then there exists an $(s+1)$-element subset $A(F)$ of $F$ such that $$\{I: A(F)\subset I\subsetneq F\}\subset \mathcal{I}.$$
\end{thm}

We use $[n]$ to denote the set $\{1, \ldots , n\}$ and use $\binom{[n]}{r}$ to denote the collection of all $r$-element subsets of $[n]$.  Let $\hf\subset \binom{[n]}{r}$ be a hypergraph. We call $\hf$ $r$-partite  if there exists a partition $[n]=X_1\cup \cdots \cup X_r$ such that $|F\cap X_i|=1$ for all $F\in \hf$ and $i\in \{1,2,\ldots,r\}$.

We adopt the statement of F\"{u}redi's structure theorem given by Frankl and Tokushige in \cite{ft}. For clarity purpose, we recall some definitions from \cite{ft}. Let $\mathcal{F}\subset \binom{[n]}{r}$ be an $r$-partite hypergraph with partition $[n]=X_1\cup\cdots \cup X_r$. For any $F\in \hf$, define the {\it restriction} of $\hf$ on $F$ by $$\mathcal{I}(F, \hf)=\{F'\cap F: F'\in \hf\setminus\{F\}\}.$$
A set of $p$ hyperedges $F_1,\ldots,F_p$ in $\mathcal{F}$ is called a {\it $p$-sunflower} if $F_i\cap F_j = C$ for every  $1\leq i<j\leq p$ and some set $C$.  The set $C$ is called {\it center} of the $p$-sunflower.

F\"{u}redi \cite{furedi} proved the following fundamental result, which was conjectured by Frankl. It roughly says that every $r$-uniform hypergraph $\hf$ contains a large $r$-partite subhypergraph $\hf^*$ satisfying that  $\mathcal{I}(F, \hf^*)$ is isomorphic to $\mathcal{I}(F', \hf^*)$ for any $F,F'\in \hf^*$.

\begin{thm}[F\"{u}redi \cite{furedi}]\label{thm4}
For positive integers $r$ and $p$, there exists a positive constant $c=c(r,p)$ such that every hypergraph $\mathcal{F}\subset \binom{[n]}{r}$ contains an $r$-partite subhypergraph $\mathcal{F}^*$ with partition $[n]=X_1\cup\cdots \cup X_r$ satisfying (i)-(iv).

(i) $|\mathcal{F}^*|\geq c|\mathcal{F}|$.

(ii) For any $F_1,F_2\in \hf^*$, $\mathcal{I}(F_1, \hf^*)$ is isomorphic to $\mathcal{I}(F_2, \hf^*)$.

(iii) For  $F\in \hf^*$, $\mathcal{I}(F, \hf^*)$ is intersection closed.

(iv) For  $F\in \hf^*$ and every $I\in \mathcal{I}(F,\hf^*)$, $I$ is the center of a $p$-sunflower in $\hf^*$.
\end{thm}

We need the following two results. The first one is due to Deza, Erd\H{o}s and Frankl \cite{deza}.

\begin{lem}[Deza, Erd\H{o}s and Frankl \cite{deza}]\label{lem-2}
Suppose that $\{E_1,\ldots,E_{r+1}\}$ and $\{F_1,\ldots,F_{r+1}\}$ are both $(r+1)$-sunflowers in an $r$-uniform hypergraphs with center $C_1$ and $C_2$, respectively. Then there exist $i$ and $j$ such that $F_i\cap F_j=C_1\cap C_2$.
\end{lem}

The second one is due to Zykov \cite{zykov}. He showed that the Tur\'{a}n graph maximizes the number of $s$-cliques in $n$-vertex $K_{t+1}$-free graphs for $s\leq t$.

\begin{thm}[Zykov \cite{zykov}]\label{thm7}
For $s\leq t$,
$$ex(n, K_s, K_{t+1})=\mathcal{N}(K_s, T_t(n)),$$
and $T_{t}(n)$ is the unique graph attaining the maximum number of copies of $K_s$.
\end{thm}

Let $\hf\subset \binom{[n]}{r}$ be a hypergraph and  $x\in [n]$. Define
\[
N_{\hf}(x) = \{T\colon T\cup\{x\}\in \hf\}.
\]
The degree of $x$ in $\hf$, denoted by $\deg_{\mathcal{F}}(x)$, is the cardinality of $N_{\hf}(x)$.

Now we are ready to prove Theorem \ref{thm3}.

\begin{proof}[Proof of Theorem \ref{thm3}]
Let $G$ be a $B_{r,1}$-free graph on $[n]$ with the maximum number of copies of $K_r$. Since $K_2\vee T_{r-2}(n-2)$ is $B_{r,1}$-free, we may assume that $\mathcal{N}(K_r, G)\geq \mathcal{N}(K_{r-2}, T_{r-2}(n-2))$.

Let
\[
\hf = \left\{F\in \binom{[n]}{r}\colon G[F] \mbox{ is a clique}\right\}.
\]
Clearly, $|F_1\cap F_2|\neq 1$ for any $F_1, F_2\in \mathcal{F}$ since $G$ is $B_{r,1}$-free. Now we apply Theorem \ref{thm4} with $p=r+1$ to $\mathcal{F}$ and obtain $\mathcal{F}_1=\mathcal{F}^*$ satisfying (i)-(iv). Then apply Theorem \ref{thm4} to $\mathcal{F}-\mathcal{F}_1$ to obtain $\mathcal{F}_2=(\mathcal{F}-\mathcal{F}_1)^*$, in the $i$-th step we obtain $\mathcal{F}_i=(\mathcal{F}-(\mathcal{F}_1\cup\cdots\cup \mathcal{F}_{i-1}))^*$. We stop if there is an $F_0\in \hf_i$ and an $(r-3)$-element subset $B_0$ of $F_0$ such that $B_0$ is not covered by $\mathcal{I}(F_0,\hf_i)$.  Suppose that the procedure stops in the $m$-th step. By Theorem \ref{thm4} (ii), for every $F\in \hf_m$ there is an $(r-3)$-element subset $B$ of $F$ such that $B$ is not covered by $\mathcal{I}(F,\hf_m)$.

\begin{claim}\label{claim6}
$|\mathcal{F}-(\mathcal{F}_1\cup\cdots\cup \mathcal{F}_{m-1})|\leq c'{n\choose r-3}$ for some $c'>0$.
\end{claim}

\begin{proof}
For any $F\in \hf_m$, let $B$ be an $(r-3)$-element subset of $F$ that is not covered by $\mathcal{I}(F,\mathcal{F}_m)$. Then it follows that $B\nsubseteq E\cap F$ for any $E\in \mathcal{F}_m\setminus \{F\}$, that is, $F$ is the only hyperedge in $\hf_m$ that contains $B$. Thus $|\mathcal{F}_m|\leq {n\choose r-3}$. Now by Theorem \ref{thm4} (i),
\begin{align*}
|\mathcal{F}-(\mathcal{F}_1\cup\cdots\cup \mathcal{F}_{m-1})|\leq c^{-1}|\mathcal{F}_m|\leq c'{n\choose r-3}.
\end{align*}
\end{proof}

Let $i\in \{1,2,\ldots,m-1\}$ and $F\in \hf_i$. By Theorem \ref{thm4} (iii), $\mathcal{I}(F,\hf_i)$ is intersection closed. Since $|F_1\cap F_2|\neq 1$ for any $F_1,F_2\in \hf_i$, $|I|\neq 1$ for each $I\in \mathcal{I}(F,\hf_i)$. Now apply Theorem \ref{thm5} with $s=1$ to $\mathcal{I}(F,\hf_i)$, we obtain a $2$-element subset $A(F)$ of $F$ such that $$\{I: A(F)\subset I \subsetneq F\}\subset \mathcal{I}(F, \mathcal{F}_i).$$ Let $A_1, A_2, \ldots, A_h$ be the list of 2-element sets for which $A_j=A(F)$ for some $F\in \mathcal{F}_1\cup \cdots \cup \mathcal{F}_{m-1}$. For $j=1,\ldots,h$, let $$\mathcal{H}_j=\{F\in \mathcal{F}_1\cup \cdots \cup \mathcal{F}_{m-1}: A(F)=A_j\}$$
and
\[
V(\mathcal{H}_j) =\bigcup_{F\in \mathcal{H}_j} F.
\]

\begin{claim}\label{claim7}
$V(\mathcal{H}_1), \ldots, V(\mathcal{H}_h)$ are pairwise disjoint.
\end{claim}

\begin{proof}
Suppose for contradiction that $|V(\mathcal{H}_1)\cap V(\mathcal{H}_2)|\geq 1$. It follows that there exist $F_1\in \mathcal{H}_1$ and $F_2\in \mathcal{H}_2$ such that $|F_1\cap F_2|\geq 1$. Then we can find two sets $C_1$ and $C_2$ satisfying $A_1\subset C_1 \subsetneq F_1$, $A_2\subset C_2 \subsetneq F_2$ and $|C_1\cap C_2|=1$ in the following way. If $|A_1\cap A_2|=1$, then let $C_1=A_1$ and $C_2=A_2$. If $A_1\cap A_2=\emptyset$, then let  $C_1=A_1\cup \{x\}$ and $C_2=A_2\cup \{x\}$ for some $x\in F_1\cap F_2$.

Since $F_1\in \mathcal{F}_i$ for some $i\in \{1, \ldots, m-1\}$ and $$C_1\in \{I: A_1\subset I \subsetneq F_1\}\subset \mathcal{I}(F_1, \mathcal{F}_i),$$
by Theorem \ref{thm4} (iv) $C_1$ is the center of an $(r+1)$-sunflower in $\mathcal{F}_i$. Therefore $C_1$ is the center of an $(r+1)$-sunflower in $\mathcal{F}$. Similarly, $C_2$ is also the center of an $(r+1)$-sunflower in $\mathcal{F}$. By Lemma \ref{lem-2}, there exist $F_1', F_2'\in \mathcal{F}$ satisfying $|F_1'\cap F_2'|=|C_1\cap C_2|=1$, which contradicts the fact that $|F_1\cap F_2|\neq 1$ for any $F_1, F_2\in \mathcal{F}$. Thus the claim holds.
\end{proof}

Assume that $A_i=\{u_i, v_i\}$ for $i=1,\ldots, h$. Let $G_i$ be the graph on the vertex set $V(\mathcal{H}_i)$
with the edge set
\[
E(G_i) =\left\{uv\colon \{u,v\}\subset F\in \mathcal{H}_i\right\}.
\]
Obviously, $G_i$ is a subgraph of $G$ and $vu_i, vv_i, u_iv_i \in E(G_i)$ for each $v\in V(\mathcal{H}_i)\setminus A_i$.

\begin{claim}\label{claim8}
$G_i- A_i$ is $K_{r-1}$-free for  $i=1,\ldots, h$.
\end{claim}
\begin{proof}
By symmetry, we only need to show that $G_1- A_1$ is $K_{r-1}$-free. Suppose for contradiction that $\{a_1, a_2, \ldots, a_{r-1}\}\subset V(G_1)\setminus\{u_1, v_1\}$ induces a copy of $K_{r-1}$ in $G_1- A_1$. Since $u_1a_j\in E(G_1)$ for each $j=1,\ldots, r-1$,  $\{u_1, a_1, a_2, \ldots, a_{r-1}\}$ induces a copy of $K_{r}$ in $G$. Note that $A_1=\{u_1,v_1\}$ is the center of an $(r+1)$-sunflower in $\mathcal{F}$. Let $F_1,F_2,\ldots,F_{r+1}$ be such a sunflower with center $A_1$. Then there exists some $F_j$ with $(F_j\setminus A_1)\cap \{a_1, a_2, \ldots, a_{r-1}\}=\emptyset$. It follows that $F_j\cap\{u_1, a_1, a_2, \ldots, a_{r-1}\}=\{u_1\}$. By the definition of $\mathcal{F}$, the subgraph of $G$ induced by $F_j\cup \{u_1, a_1, a_2, \ldots, a_{r-1}\}$ contains $B_{r,1}$. This contradicts the fact that $G$ is $B_{r,1}$-free and the claim follows.
\end{proof}

Let $x_i=|V(\mathcal{H}_i)|$ for $i=1,2,\ldots, h$ and assume that $x_1\geq x_2\geq\cdots\geq x_h$. By Claim \ref{claim7}, $x_1+\cdots +x_h\leq n$.

\begin{claim}\label{claim33}
$x_1\geq n-c''$, for some constant $c''>0$.
\end{claim}

\begin{proof}
By Claim \ref{claim8} and Theorem \ref{thm7}, the number of copies of $K_{r-2}$ in $G_i - A_i$ is at most
$\mathcal{N}(K_{r-2}, T_{r-2}(x_i-2))$. It follows that
\[
|\mathcal{H}_i| \leq \mathcal{N}(K_{r-2}, T_{r-2}(x_i-2))
\]
for each $i=1,\ldots,h$. By Claims \ref{claim6} and \ref{claim7},
\begin{align}\label{ineq-3}
\mathcal{N}(K_r, G)&= |\mathcal{F}-(\mathcal{F}_1\cup\cdots \cup\mathcal{F}_{m-1})|+|(\mathcal{F}_1\cup\cdots \cup\mathcal{F}_{m-1})|\nonumber\\[5pt]
&= |\mathcal{F}-(\mathcal{F}_1\cup\cdots \cup \mathcal{F}_{m-1})|+|\mathcal{H}_1|+\cdots+|\mathcal{H}_h|\nonumber\\[5pt]
&\leq c'{n\choose r-3}+ \sum_{i=1}^h \mathcal{N}(K_{r-2}, T_{r-2}(x_i-2)).
\end{align}
Since
\begin{align*}
\mathcal{N}(K_{r-2}, T_{r-2}(x_i-2))\leq \left(\frac{x_i-2}{r-2}\right)^{r-2},
\end{align*}
we have
\begin{align}\label{ineq-4}
\mathcal{N}(K_r, G)\leq & c'{n\choose r-3}+ \sum_{i=1}^h\left(\frac{x_i-2}{r-2}\right)^{r-2}\nonumber\\[5pt]
\leq & c'{n\choose r-3}+\sum_{i=1}^h (x_i-2) \cdot \frac{(x_1-2)^{r-3}}{(r-2)^{r-2}}\nonumber\\[5pt]
\leq & c'{n\choose r-3}+\frac{(x_1-2)^{r-3}(n-2)}{(r-2)^{r-2}}.
\end{align}
By our assumption,
\begin{align}\label{ineq-2}
\mathcal{N}(K_r, G)\geq \mathcal{N}(K_{r-2}, T_{r-2}(n-2)) \geq \left( \frac{n-r}{r-2}\right)^{r-2}.
\end{align}
Combining \eqref{ineq-4} and \eqref{ineq-2}, we obtain that
\[
1\leq c'{n\choose r-3}{\left( \frac{r-2}{n-r}\right)^{r-2}} + \frac{n-2}{n-r}\cdot \left(\frac{x_1-2}{n-r}\right)^{r-3}.
\]
Since $n$ is sufficiently large, we get $x_1\geq \frac{n}{2}+r$.

Let $n_1,n$ be two integers with $0<n_1<n$ and let $H$ be an $r$-partite Tur\'{a}n graph on $n$ vertices with vertex classes $V_1, V_2,\ldots,V_r$. Then there exist partitions $V_j =V_{j,1} \cup V_{j,2}$ for each $j=1,2,\ldots,r$ such that
\[
\sum_{j=1}^r |V_{j,1}|=n_1
\]
and both $H[\cup_{j=1}^r V_{j,1}]$ and $H[\cup_{j=1}^r V_{j,2}]$ are Tur\'{a}n graphs.
By considering the edges between  $\cup_{j=1}^r V_{j,1}$ and $\cup_{j=1}^r V_{j,2}$, it is easy to see that
\begin{align}\label{ineq-1}
\mathcal{N}(K_r, T_r(n))> \mathcal{N}(K_r, T_r(n_1)) +\mathcal{N}(K_r, T_r(n-n_1)) + \left\lfloor \frac{n-n_1}{r}\right\rfloor \cdot \mathcal{N}(K_{r-1}, T_r(n_1)).
\end{align}
Apply the inequality \eqref{ineq-1} inductively, we have
\begin{align}\label{ineq-5}
\sum_{i=2}^h  \mathcal{N}(K_{r-2}, T_{r-2}(x_i-2)) < \mathcal{N}(K_{r-2}, T_{r-2}(n-x_1)).
\end{align}
By \eqref{ineq-3} and \eqref{ineq-5}, we see that
\begin{align*}
\mathcal{N}(K_r, G)\leq & c'{n\choose r-3}+ \mathcal{N}(K_{r-2}, T_{r-2}(x_1-2))+\mathcal{N}(K_{r-2}, T_{r-2}(n-x_1)).
\end{align*}
Apply the inequality \eqref{ineq-1} again, we obtain that
\begin{align}\label{ineq-6}
\mathcal{N}(K_r, G)&\leq c'{n\choose r-3}+ \mathcal{N}(K_{r-2}, T_{r-2}(n-2)) - \left\lfloor \frac{n-x_1+2}{r}\right\rfloor \cdot \mathcal{N}(K_{r-3}, T_{r-2}(x_1-2))\nonumber\\[5pt]
&\leq \mathcal{N}(K_{r-2}, T_{r-2}(n-2))+c'{n\choose r-3} - \frac{n-x_1-r}{r} \cdot (r-2) \left( \frac{x_1-r}{r-2}\right)^{r-3}.
\end{align}
It follows from  \eqref{ineq-2} and \eqref{ineq-6} that
\[
c'{n\choose r-3}\geq  \frac{n-x_1-r}{r} \cdot (r-2) \left( \frac{x_1-r}{r-2}\right)^{r-3}.
\]
Since $x_1>\frac{n}{2}+r$, we arrive at
\[
c'{n\choose r-3}\geq  \frac{n-x_1-r}{r\cdot 2^{r-2}} \cdot (r-2) \left( \frac{n}{r-2}\right)^{r-3}.
\]
It follows that $x_1\geq n-c''$ for some $c''>0$.
\end{proof}

Let us define
\[
\mathcal{K}=\left\{F\in \mathcal{F}:\begin{array}{l}A_1\subset F \text{ and for each } I \text{ with } A_1\subset I\subsetneq F,\\
I\text{ is the center of an }(r+1)\text{-sunflower in } \mathcal{F} \end{array}\right\}.
\]
Obviously, we have $\mathcal{H}_1\subset \mathcal{K}$. Define $$\mathcal{A}=\{F\in\mathcal{F}: A_1\subset F, F\notin \mathcal{K}\}$$ and $$\mathcal{B}=\mathcal{F}-\mathcal{K}-\mathcal{A}.$$
Note that $V(\hk)=\cup_{F\in \hk} F$ and $V(\mathcal{B})=\cup_{F\in \mathcal{B}} F$. We claim that $V(\mathcal{K})\cap V(\mathcal{B})=\emptyset$. Otherwise, there exist $F_1\in \hk$ and $F_2\in \mathcal{B}$ with $|F_1\cap F_2|\geq 1$. Note that $A_1\subset F_1$ and $A_1\not\subset F_2$. If $F_2\cap A_1= \emptyset$, let $C=A_1\cup \{x\}$ with $x\in F_1\cap F_2$. If $F_2\cap A_1\neq \emptyset$, then let $C = A_1$. It is easy to see that $|C\cap F_2|=1$ in both of the two cases. Clearly, we have $A_1\subset C\subsetneq F_1$. By the definition of $\hk$, $C$ is center of an $(r+1)$-sunflower in $\hf$. Let $E_1,E_2,\ldots,E_{r+1}$ be such a sunflower. Since $|F_2\setminus C|< r$, there exists some $E_j$ such that $(E_j\setminus C)\cap (F_2\setminus C) =\emptyset$. Then we have $|E_j\cap F_2| =|C\cap F_2|=1$, a contradiction. Thus
$V(\mathcal{K})\cap V(\mathcal{B})=\emptyset$.

By Claim \ref{claim33}, we have
\begin{align}\label{ineq-7}
|V(\mathcal{B})|\leq n-V(\mathcal{K})\leq n-V(\mathcal{H}_1)\leq c''.
\end{align}
 Let $\mathcal{C} = \{F\in \mathcal{A}\colon F \cap V(\mathcal{B})=\emptyset\}$,
$\mathcal{K}' =\mathcal{K} \cup \mathcal{C}$ and $\mathcal{A}' =\mathcal{A} \setminus \mathcal{C}$. Clearly, $V(\mathcal{K}') \cap V(\mathcal{B})=\emptyset$, $F\cap V(\mathcal{K}') \supset A_1$ and $F\cap V(\mathcal{B})\neq \emptyset$ for each $F\in \mathcal{A}'$.

\begin{claim}\label{claim9}
$\mathcal{B}=\emptyset$.
\end{claim}

\begin{proof}
Suppose for contradiction that there exists $B\in \mathcal{B}$. We first show that the degree of each vertex $x$ in $B$ is small.  By \eqref{ineq-7}, we have
\[
\deg_{\mathcal{B}}(x) \leq \binom{|V(\mathcal{B})|}{r-1} \leq \binom{c''}{r-1}.
\]
Note that $A_1\subset F$ for any $F\in \hf\setminus \mathcal{B}$ and $|F\cap F'|\neq 1$ for any $F,F'\in \hf$. We have $A_1\subset B'$ and $|B'\cap B|\geq 2$ for any $B'\in \hf\setminus \mathcal{B}$ with $x\in B'$. Thus, the number of hyperedges containing $x$ in $\hf\setminus \mathcal{B}$ is at most $|B\setminus \{x\}|\cdot\binom{n}{r-4}=(r-1)\binom{n}{r-4}$. Therefore,
\[
\deg_{\hf}(x)\leq \deg_{\mathcal{B}}(x) +(r-1)\binom{n}{r-4} \leq \binom{c''}{r-1}+(r-1)\binom{n}{r-4}.
\]

Let $u\in V(\mathcal{K}')\setminus A_1$ be the vertex with $$\deg_{\mathcal{K}'}(u)= \max \left\{\deg_{\mathcal{K}'}(v)\colon v\in V(\mathcal{K}')\setminus A_1 \right\}.$$
We show that $\deg_{\mathcal{K}'}(u)\geq c'''n^{r-3}$ for some constant $c'''>0$. Since $F\cap V(\mathcal{B})\neq \emptyset$ for each $F\in \mathcal{A}'$, we have
\[
|\mathcal{A}'|+|\mathcal{B}| \leq \sum_{v\in V(\mathcal{B})}\deg_{\hf}(v).
\]
If $\deg_{\mathcal{K}'}(u)=o(n^{r-3})$, then
\begin{align*}
\mathcal{N}(K_r, G)&=|\mathcal{K}'|+|\mathcal{A}'|+|\mathcal{B}|\\[6pt]
&\leq \frac{1}{r-2}\sum_{v\in V(\mathcal{K}')\setminus A_1}\deg_{\mathcal{K}'}(v)+\sum_{v\in V(\mathcal{B})}\deg_{\hf}(v)\\[6pt]
&\leq o(n^{r-2})+c''\left((r-1)\binom{n}{r-4}+\binom{c''}{r-1}\right),
\end{align*}
which contradicts the assumption that $\mathcal{N}(K_r, G)\geq \mathcal{N}(K_{r-2}, T_{r-2}(n-2))$.
Thus $\deg_{\mathcal{K}'}(u)\geq c'''n^{r-3}$ for some constant $c'''>0$.

Since $n$ is sufficiently large, for each $x\in B$ we have
\[
\deg_{\hf}(u) \geq \deg_{\mathcal{K}'}(u) \geq c'''n^{r-3} > \deg_{\hf}(x).
\]
We claim that there exists $x_0\in B$ such that $ux_0$ is not an edge of $G$. Otherwise, if $ux\in E(G)$ for all $x\in B$, then $\{u\}\cup T$ induces a copy of $K_r$ in $G$ for any $T\in \binom{B}{r-1}$. Since $\deg_{\mathcal{K}'}(u) \geq c'''n^{r-3}$, there exists an hyperedge $K$ in $\mathcal{K}'$ containing $u$. Recall that $V(\mathcal{K}')\cap V(\mathcal{B})=\emptyset$. Then $\{u\}\cup T \cup K$ induces a copy of $B_{r,1}$ in $G$, a contradiction. Thus, there exists $x_0\in B$ such that $ux_0$ is not an edge of $G$.

Now let $G'$ be a graph obtained from $G$ by deleting  edges incident to $x_0$ and adding  edges in $\{x_0w\colon w\in N(u)\}$. We claim that $G'$ is $B_{r,1}$-free. Otherwise, there exist two copies $K,K'$ of $K_r$ in $G'$ with $V(K)\cap V(K')=\{y\}$ for some $y\in V(G')$. Since $G$ is $B_{r,1}$-free, we may assume that $x_0\in V(K)$. If  $u\notin V(K')$, then $V(K)\cup V(K')\setminus \{x_0\} \cup \{u\}$ induces a copy of $B_{r,1}$ in $G$, a contradiction. If $u\in V(K')$, then $y\neq x_0$ since $x_0y$ is not an edge in $G'$. Moreover, $V(K')\notin \mathcal{B}$ and $V(K)\setminus \{x_0\}\cup \{u\}\notin \mathcal{B}$ since $u\in V(\mathcal{K}')$. By the definition of $\mathcal{K}'$ and $\mathcal{A}'$, we see that both $V(K')$ and $V(K)\setminus \{x_0\}\cup \{u\}$ contains $A_1$. But now we have $V(K)\cap V(K')\supset A_1$ since $u,x_0\notin A_1$, which contradicts our assumption that $V(K)\cap V(K')=\{y\}$. Thus $G'$ is $B_{r,1}$-free.

Since $\deg_{\hf}(u)  > \deg_{\hf}(x_0)$, we have
\[
\mathcal{N}(K_r, G')=\mathcal{N}(K_r, G)-\deg_{\hf}(x_0)+\deg_{\hf}(u)>\mathcal{N}(K_r, G),
\]
which contracts the maximality of the number of copies of $K_r$ in $G$. Thus, the claim follows.
\end{proof}

By Claim \ref{claim9}, $A_1$ is contained in every hyperedge of $\mathcal{F}$.  Recall that $A_1=\{u_1,v_1\}$. It follows that $xu_1,xv_1\in E(G)$ for any $x\in V(G)\setminus A_1$. We claim that $G\setminus A_1$ is $K_{r-1}$-free. Otherwise, let $\{a_1, a_2, \ldots, a_{r-1}\}\subset V(G)\setminus A_1$ be a set that induces a copy of $K_{r-1}$ in $G- A_1$. Since $u_1a_j\in E(G)$ for each $j=1,\ldots, r-1$,  $\{u_1, a_1, a_2, \ldots, a_{r-1}\}$ induces a copy of $K_{r}$ in $G$. Note that $A_1$ is the center of an $(r+1)$-sunflower in $\mathcal{F}$. Let $F_1,F_2,\ldots,F_{r+1}$ be such a sunflower with center $A_1$. Then there exists some $F_j$ with $(F_j\setminus A_1)\cap \{a_1, a_2, \ldots, a_{r-1}\}=\emptyset$. It follows that $F_j\cap\{u_1, a_1, a_2, \ldots, a_{r-1}\}=\{u_1\}$. By the definition of $\mathcal{F}$, the subgraph of $G$ induced by $F_j\cup \{u_1, a_1, a_2, \ldots, a_{r-1}\}$ contains $B_{r,1}$, a contradiction. Thus $G- A_1$ is $K_{r-1}$-free.

By Theorem \ref{thm7}, there are at most $\mathcal{N}(K_{r-2}, T_{r-2}(n-2))$ copies of $K_{r-2}$ in $G- A_1$ and Tur\'{a}n graph $T_{r-2}(n-2)$ is the unique graph attaining the maximum number. Thus, the number of $K_r$ in $G$ is at most $\mathcal{N}(K_{r-2}, T_{r-2}(n-2))$ and $K_2\vee T_{r-2}(n-2)$ is the unique graph attaining the maximum number of copies of $K_r$.
\end{proof}

Now we prove Theorem \ref{thm8} using F\"{u}redi's structure theorem.

\begin{proof}[Proof of Theorem \ref{thm8}]
Let $G$ be a $B_{r,0}$-free graph on vertex set $[n]$ and let
\[
\hf = \left\{F\in \binom{[n]}{r}\colon G[F] \mbox{ is a clique}\right\}.
\]
Since $G$ is $B_{r,0}$-free, $\mathcal{F}$ is an intersecting family. We apply Theorem \ref{thm4} with $p=r+1$ to $\mathcal{F}$ and obtain $\mathcal{F}^*$. Let $\mathcal{I}=\mathcal{I}(F, \mathcal{F}^*)$ for some fixed $F\in\mathcal{F}^*$. From Theorem \ref{thm4} (iv) and Lemma \ref{lem-2}, we have $|I\cap I'|\geq 1$ for any $I, I'\in \mathcal{I}$. Let $I_0$ be a  minimal set in $\mathcal{I}$.  Since $\mathcal{I}$ is intersection closed, $I_0\subset I$ for all $I\in\mathcal{I}$. Otherwise we have $I_0\cap I\in \mathcal{I}$ and $|I\cap I_0|<|I_0|$, which contracts the minimality of $I_0$. Now we distinguish two cases.

{ Case 1.} $|I_0|=1$. Let $I_0=\{v\}$. By Theorem \ref{thm4} (iv), $\{v\}$ is  center of an $(r+1)$-sunflower in $\mathcal{F}^*$. Let $F_1,F_2,\ldots, F_{r+1}$ be hyperedges in such an $(r+1)$-sunflower. If there is a hyperedge $F$ in $\hf$ with $v\notin F$, then it is easy to find some $j$ such that $F_j\cap F=\emptyset$, which contradicts the fact that $\mathcal{F}$ is an intersecting family. Thus, $v$ is contained in every hyperedge of $\mathcal{F}$. Let $G'= G[N(v)]$. Since each copy of $K_r$ in $G$ contains $v$, $G'$ is $K_{r}$-free. By Theorem \ref{thm7}, we have
\begin{align*}
\mathcal{N}(K_r, G)\leq \mathcal{N}(K_{r-1}, G')\leq \mathcal{N}(K_{r-1}, T_{r-1}(n-1)),
\end{align*}
and the equality holds if and only if  $G\cong K_1\vee T_{r-1}(n-1)$.

{ Case 2.} $|I_0|\geq 2$. We claim that $F\setminus I_0$ is not covered by $\mathcal{I}$. Otherwise, assume that $F\setminus I_0\subset I^*$ for some $I^*\in \mathcal{I}$. Since $I_0\subset I$ for all $I\in \mathcal{I}$,  we have $I_0\subset I^*$. It follows that $I^*=F$, which contradicts the fact that $F\notin \mathcal{I}$. Hence $F\setminus I_0$ is not covered by $\mathcal{I}$. It follows that $F$ is the only hyperedge in $\hf^*$ containing $F\setminus I_0$. Theorem \ref{thm4} (ii) shows that $\mathcal{I}(F, \mathcal{F}^*)$ is isomorphic to $\mathcal{I}(F', \mathcal{F}^*)$ for any $F,F'\in\mathcal{F}^*$. For any $E\in\mathcal{F}^*$, there is an $(r-|I_0|)$-element subset $T$ of $E$ such that $E$ is the only hyperedge in $\hf^*$ containing $T$. Since $|I_0|\geq 2$, we have  $|\mathcal{F}^*|\leq {n\choose r-2}$. By Theorem \ref{thm4} (i), for sufficiently large $n$, we have
\begin{align*}
\mathcal{N}(K_r, G)=|\mathcal{F}|\leq c^{-1}|\mathcal{F}^*|\leq c^{-1}{n\choose r-2}< \mathcal{N}(K_{r-1}, T_{r-1}(n-1)).
\end{align*}
This completes the proof.
\end{proof}

\section{Bounds on $ex(n,K_r,B_{r,s})$ for general $r$ and $s$}

Let $B^{(r)}_s$ be  an $r$-uniform hypergraph consisting of two hyperedge that share exactly $s$ vertices.
Let $ex_r(n,B^{(r)}_s)$  denote the maximum number of hyperedges in an $r$-uniform $B^{(r)}_s$-free hypergraph on $n$ vertices. In \cite{ff85}, Frankl and F\"{u}redi proved that
\begin{thm}[Frankl and F\"{u}redi \cite{ff85}]\label{ffthm}
For $r\geq 2s+2$ and $n$ sufficiently large,
\[
ex_r(n,B^{(r)}_s) =\binom{n-s-1}{r-s-1}.
\]
For $r\leq 2s+1$,
\[
ex_r(n,B^{(r)}_s) =O(n^s).
\]
\end{thm}

Now we prove Theorem \ref{thm9} by using Theorem \ref{ffthm}.

\begin{proof}[Proof of Theorem \ref{thm9}]
Notice that $ex(n,K_r,B_{r,s})\leq ex_r(n,B^{(r)}_s)$, by Theorem \ref{ffthm} we have
\begin{align}\label{ffupbound}
ex(n,K_r,B_{r,s})= O(n^{\max\{s,r-s-1\}}).
\end{align}

For $r\geq 2s+1$, it is easy to see that $K_{s+1}\vee T_{r-s-1}(n-s-1)$ is a $B_{r,s}$-free graph. Then
\[
ex(n,K_r,B_{r,s}) \geq \hn(K_{r-s-1},T_{r-s-1}(n-s-1)).
\]
By \eqref{ffupbound}, we have $ex(n,K_r,B_{r,s})=\Theta(n^{r-s-1})$.

For $r\leq 2s$, by using $s$-sum-free partition of $r$, we give a lower bound construction as follows. Let $P=(a_1,a_2,\ldots,a_t)$ be an $s$-sum-free partition of $r$.  Define a graph $G_P$ on the vertex set $V(G)=X_1\cup X_2\cup \ldots X_t$ with $X_i=\lfloor n/t\rfloor$ or $\lceil n/t\rceil$ for each $i=1,2,\ldots,t$. Let $G_P[X_i]$ be the union of $|X_i|/a_i$ vertex-disjoint copies of $K_{a_i}$ for each $i=1,2,\ldots,t$ and $G_P[X_i,X_j]$ be a complete bipartite graph for $1\leq i<j\leq t$.

We claim that $G_P$ is $B_{r,s}$-free. Let $K, K'$ be two copies of $K_r$ in $G_P$. Since $G_P[X_i]$ is a union of vertex-disjoint copies of $K_{a_i}$, we have $|V(K)\cap X_i|\leq  a_i$ and $|V(K')\cap X_i|\leq  a_i$. It follows that  $|V(K)\cap X_i|= a_i$ and $|V(K')\cap X_i|= a_i$ because of $a_1+\cdots+a_t=r$. Since  $P$ is $s$-sum-free, we conclude that $|V(K)\cap V(K')|\neq s$. Thus, $G_P$ is $B_{r,s}$-free. Moreover,
 \[
 \hn(K_r,G_P) = \prod_{i=1}^t \left\lfloor\frac{n}{ta_i}\right\rfloor \approx  \left(t^t\prod_{i=1}^t a_i\right)^{-1} n^t.
 \]
Note that $\beta_{r,s}$ is defined to be the maximum length $t$ in an $s$-sum-free partition of $r$. Thus, the construction gives that
\[
ex(n,K_r,B_{r,s}) =\Omega(n^{\beta_{r,s}})
\]
for $r\leq 2s$. This completes the proof.
\end{proof}

\section{Bounds on $ex(n,K_4, B_{4,2})$}

In this section, we derive an upper bound on $ex(n,K_4, B_{4,2})$ by utilizing the graph removal lemma.

Let $G=(V, E)$ be a graph. For any $E'\subset E(G)$, let $G[E']$ denote the subgraph of $G$ induced by the edge set $E'$, and let $G- E'$ denote the subgraph of $G$ induced by $E(G)\setminus E'$. We use $v(G)$ to denote the number of vertices in a graph $G$.
\begin{lem}[Graph removal lemma \cite{fox}]
For any graph $H$ and any $\epsilon>0$, there exists $\delta>0$ such that any graph on $n$ vertices which contains at most $\delta n^{v(H)}$ copies of $H$ may be made $H$-free by removing at most $\epsilon n^2$ edges.
\end{lem}

\begin{proof}[Proof of Theorem \ref{thm2}]

The lower bound in the theorem is due to the following construction. Suppose that $n=6m+t$ with $t\leq 5$, let $G^*$ be a graph on $n$ vertices consisting of a set $V$ of size $3m$, whose induced subgraph is a union of $m$ disjoint copies of triangles, and an independent set $U$ of size $3m+t$ as well as all the edges between $V$ and $U$. Then, it is easy to see that $G^*$ is $B_{4,2}$-free and
\begin{align*}
\mathcal{N}(K_4, G^*)= m(3m+t)\geq\frac{n^2}{12}-2.
\end{align*}
Thus, we are left with the proof of the upper bound.

Let $G$ be a $B_{4,2}$-free graph on $n$ vertices. We may further assume that each edge of $G$ is contained in at least one copy of $K_4$.

\begin{claim}\label{claim10}
There is a subset $E'\subset E(G)$ with $|E'|=o(n^2)$ such that $G'=G- E'$ is $K_5$-free, and $\mathcal{N}(K_4, G)= \mathcal{N}(K_4, G')+o(n^2)$.
\end{claim}
\begin{proof}
For any edge $e$ in $G$, there is at most one copy of $K_5$ containing $e$, since otherwise we shall find a copy of $B_{4,2}$. Thus, the number of $K_5$ in $G$ is $O(n^2)=o(n^5)$. By the graph removal lemma, we can delete $o(n^2)$ edges to make $G$ $K_5$-free. Let $E'$ be the set of the deleted edges.

Note that the edge deletion is to remove the copy of $K_5$ in $G$, so the deleted edges are contained in some $K_5$ in $G$. Moreover, for any  $e\in E'$, there is exactly one copy of $K_5$ in $G$ containing $e$. We denote it by $K$. Then  each copy of $K_4$ containing $e$ is a subgraph of $K$, otherwise we shall find a copy of $B_{4,2}$. Thus, there are at most three copies $K_4$ in $G$ containing $e$. Thus, edge deletion reduces at most $o(n^2)$ copies of $K_4$.
\end{proof}

Let $R$ be a subset of $E(G')$ consisting of all the edges contained in at least two copies of $K_4$ in $G'$, and let $B=E(G')\setminus R$.

\begin{claim}\label{claim11}
There is a subset $T\subset B$ with $|T|=o(n^2)$ such that $G'[B\setminus T]$ is $K_4$-free, and $\mathcal{N}(K_4,G')= \mathcal{N}(K_4,G'- T)+o(n^2)$.
\end{claim}

\begin{proof}
By the definition of the set $B$, each edge in $B$ is contained in at most one copy of $K_4$ in $G'$. Thus, the number of copies of $K_4$ in $G'[B]$ is at most $O(n^2)=o(n^4)$. By the graph removal lemma, we can delete $o(n^2)$ edges to make $G'[B]$ $K_4$-free. Moreover, for any deleted edge $e$, since $e\in B$ it follows that $e$ is contained in exactly one copy of $K_4$ in $G'$. Thus, the edge deletion decreases at most $o(n^2)$ copies of $K_4$.
\end{proof}

Let $G^*=G'- T$, $B^*=B\setminus T$.  Then the edge set of $G^*$ consists of $R$ and $B^*$, and $G^*[B^*]$ is $K_4$-free. In Claim \ref{claim11}, the edge deletion is to remove the copy of $K_4$ in $G'[B]$, and each deleted edge is contained in exactly one copy of $K_4$ in $G'[B]$.
Then each edge in $R$  is still contained in at least two copies of $K_4$ in $G^*$ and every edge in $B^*$ is contained in at most one copy of $K_4$ in $G^*$. We say a copy of $K_4$ in $G^*$  {\it right-colored} if three of its edges form a triangle in $G^*[R]$ and the other three edges form a star in $G^*[B^*]$.

\begin{claim}
 All the copies of $K_4$ in $G^*$ are right-colored.
\end{claim}
 \begin{proof}
 Suppose that $S=\{v_1, v_2, v_3, v_4\}$ induces a copy of $K_4$ in $G^*$. Clearly, at least one edge in $G^*[S]$ is contained in $R$. Without loss of generality, assume that $v_1v_2$ be such an edge. Since $v_1v_2$ is contained in at least two copies of $K_4$ in $G^*$, assume that $G^*[\{v_1,v_2,v_s,v_t\}]$ be another copy of $K_4$ containing $v_1v_2$. If $\{v_s,v_t\}\cap \{v_3,v_4\}=\emptyset$, then we find a copy of $B_{4,2}$ in $G^*$, a contradiction. Thus, we have $|\{v_s,v_t\}\cap \{v_3,v_4\}|=1$. Assume that $v_s=v_3$, then both $v_1v_3$ and $v_2v_3$ are  contained in at least two copies of $K_4$. It follows that $v_1v_3$ and $v_2v_3$ are edges in $R$. Thus, there are three edges in $G^*[S]$ belonging to $R$ that form a triangle in $G^*$.

Next we show that $v_1v_4,v_2v_4$ and $v_3v_4$ are all edges in $B^*$. If not, assume that $v_3v_4\in R$. Then, all the copies of $K_4$ containing $v_1v_2$ should also contain $v_3$ or $v_4$, otherwise we shall find a copy of $B_{4,2}$. Without loss of generality, assume that all the copies of $K_4$ containing $v_1v_2$ contain $v_3$ as well. Let $G^*[\{v_1,v_2,v_3,v_4\}]$ and $G^*[\{v_1,v_2,v_3,v_5\}]$ be two such copies of $K_4$. Similarly, all the copies of $K_4$ containing $v_3v_4$ should also contain $v_1$ or $v_2$. Without loss of generality, assume that $G^*[\{v_3,v_4,v_1,v_2\}]$ and $G^*[\{v_3,v_4,v_1,v_6\}]$ be two such copies of $K_4$. Clearly, we have $v_5\neq v_6$ for $G^*$ is $K_5$-free. However, at this time both $G^*[\{v_1,v_3,v_4,v_6\}]$ and $G^*[\{v_1,v_3,v_2,v_5\}]$ form a copy of $K_4$, which implies $G^*[\{v_1, v_2, v_3, v_4, v_5, v_6\}]$ contains a copy of $B_{4,2}$, a contradiction. Thus, $v_3v_4\in B^*$.

 Similarly, we can deduce that $v_1v_4$ and $v_2v_4$ are edges in $B^*$. Therefore, $G^*[S]$ is right-colored and the claim holds.
\end{proof}

Since $G^*[B^*]$ is $K_4$-free, by Tur\'{a}n theorem \cite{turan} there are at most $\frac{n^2}{3}$ edges in $G^*[B^*]$. Moreover, since  all the copies of $K_4$ in $G^*$ are right-colored, it follows that each  copy of $K_4$ in $G^*$  contains three edges in $B^*$. Thus, we have
\begin{align*}
\mathcal{N}(K_4, G^*)\leq \frac{|B^*|}{3}\leq \frac{n^2}{9}.
\end{align*}
From Claims \ref{claim10} and \ref{claim11}, it follows that
\[
\mathcal{N}(K_4, G)= \mathcal{N}(K_4, G^*)+o(n^2) \leq \frac{n^2}{9}+o(n^2),
\]
which completes the proof.
\end{proof}

\vspace{10pt}\noindent
{\bf Acknowledgement.} 
The second author is supported by  Shanxi Province Science Foundation for Youths, 
PR China (No. 201801D221028).

\end{document}